\documentclass[11pt]{amsart}

\usepackage[all]{xy}
\usepackage{amsmath}
\usepackage{amsfonts,mathrsfs,bbding}
\usepackage{amssymb}
\usepackage{amscd}
\usepackage{amsthm}
\usepackage[titletoc]{appendix}
\usepackage{latexsym}
\usepackage{graphicx}
\usepackage{epstopdf}
\usepackage{amsrefs}
\usepackage{mathscinet}
\usepackage{epsfig}
\usepackage{mathtools}
\usepackage[latin1]{inputenc}
\usepackage{verbatim}
\usepackage{pb-diagram}
\usepackage{lipsum}
\usepackage{faktor}
\usepackage{tikz-cd}
\usepackage{fancyhdr}
\AppendGraphicsExtensions{.gif}

\theoremstyle{definition}
\newtheorem{definition}{Definition}[section]

\newtheorem{example}{Example}[section]

\newtheorem{observation}{Observation}[section]

\theoremstyle{plain}% plain: italic text, extra space above and below; Theorem, Lemma, Corollary, Proposition, Conjecture, Criterion, Assertion
\newtheorem{theorem}{Theorem}[section]
\newtheorem{corollary}{Corollary}[theorem]
\newtheorem{lemma}[theorem]{Lemma}
\newtheorem{proposition}[theorem]{Proposition}
\def\limind{\mathop{\oalign{lim\cr\hidewidth$\longrightarrow$\hidewidth\cr}}} %%%%%%%%%%% inductive limit
\def\limproj{\mathop{\oalign{lim\cr\hidewidth$\longleftarrow$\hidewidth\cr}}} %%%%%%%%%%% projective limit

\newcommand{\compcent}[1]{\vcenter{\hbox{$#1\circ$}}}
\newcommand{\comp}{\mathbin{\mathchoice
{\compcent\scriptstyle}{\compcent\scriptstyle}
{\compcent\scriptscriptstyle}{\compcent\scriptscriptstyle}}} 
\newcommand{\op}{\operatorname }
\newcommand{\cali}{\mathscr }

\newcommand{\pe}{\widetilde}

\newcommand{\m}{\mathcal }

\title{Convergence of Tsirelson convolution systems of probability spaces}

\author{Remus Floricel}
\address{University of Regina, Department of Mathematics, Regina, SK, Canada}
\email{Remus.Floricel@uregina.ca}
\author{Patrick Melanson} 
\address{University of Regina, Department of Mathematics, Regina, SK, Canada}
\email{pdm456@uregina.ca}
 \subjclass[2020]
{Primary 28A33, 60A10, 60B15, 46L53, 46L55}
\keywords{convolution systems, continuous product of probability spaces, flow systems, projective limits of probability spaces, subproduct and product system of Hilbert spaces}
\thanks{The first author was supported by a Discovery Grant from NSERC}

\begin{document}

\maketitle
\begin{abstract}

We associate two specific projective systems of probability spaces to any Tsirelson convolution system. If the corresponding projective limit (in the sense of Bochner) exists, we say that the system is convergent in the first case and $K$-convergent in the second.
It is shown that convergent convolution systems give rise to continuous products of probability spaces, while $K$-convergent systems lead to flow systems. We then investigate the relationship between convergence and $K$-convergence, as well as their connections to two-parameter product systems of Hilbert spaces.

\end{abstract}

\section{Introduction}
Motivated by Arveson's theory of product systems and $E_0$-semigroups \cite{Arveson 89, Arveson-book}, B. Tsirelson initiated the study of two-parameter product systems of Hilbert spaces in \cite{Tsi03, Tsi04} (see also \cite{TV}) through probabilistic methods, including the analysis of stochastic processes and flows. Considering that convolution semigroups of measures can naturally be used to construct stochastic processes via Kolmogorov's extension theorem, which in turn lead to product systems, one of Tsirelson's initial steps was to introduce the concept of a convolution system of standard probability spaces as a generalization of convolution semigroups. However, unlike convolution semigroups of measures, convolution systems may not always give rise to a stochastic flow. 

One notable obstacle is that in various instances, projective limits of probability spaces may fail to exist \cite{Bochner}, making generalized Kolmogorov-type extension arguments rather challenging.
This issue has a long history: classical criteria for the existence of projective limits were established by Bochner \cite{Bochner}, and further developed by Choksi \cite{Chol}, Rao \cite{Rao71}, and Frol'{\i}k \cite{Frol}, among others. From a functional-analytic perspective, these measure-theoretic questions have natural analogues in the theory of noncommutative dynamics, notably in Arveson's framework \cite{Arveson-book}\cite{Lieb}. %and the two-parameter product systems introduced by Tsirelson and Vershik \cite{Tsi03, Tsi04, TV}. 
Later developments, such as Bhat-Skeide's work on Hilbert modules \cite{Bhat-Skeide} and the $C^*$-subproduct system approach in \cite{FK}, further connect these ideas with the theory of operator algebras. Our work positions Tsirelson convolution systems within this broader landscape, treating them as a natural setting for analyzing projective limit phenomena and their connections to stochastic flows and continuous product structures.

When a Tsirelson convolution system leads to a stochastic flow, or a flow system, as it is called in \cite[Sec. 3b]{Tsi04}, the flow system can be directly used to construct a continuous product of probability spaces (CPPS), largely due to the specific properties of the standard spaces used in this construction (see \cite[Sec. 3c]{Tsi04}). This, in turn, leads to a two-parameter product system of Hilbert spaces by passing to the corresponding $L^2$-spaces.

The main purpose of this paper is to show that continuous products of (not necessarily standard) probability spaces over a linearly ordered set can be obtained from some convolution systems, even if they may not admit flow systems. It also explores the interplay between convolution systems that admit CPPS and those that admit flow systems. This is achieved by associating two projective systems of probability spaces with a given convolution system (see Proposition \ref{projsysPst} and Proposition \ref{projsysPst2}). If the projective limits of these systems exist, they will be referred to as convergent and $K$-convergent, respectively. In these cases, they give rise to a continuous product of probability spaces and a flow system, respectively.

We show in Theorem \ref{PS} that the product system of $L^2$-Hilbert spaces of the projective CPPS associated with a convergent convolution system can be described in terms of the subproduct system of $L^2$-Hilbert spaces of the convolution system through an inductive-limit construction. We also show in Proposition \ref{Kimp} that any $K$-convergent system is convergent and find in Theorem \ref{Kimpa} the necessary and sufficient conditions for a convergent system to be $K$-convergent.\\\\
{\bf{Notation and Conventions.}} All notation used in this article is standard. We write $(\Omega, \m{F}, \mu)$ for a probability space, where $\m{F}$ is a $\sigma$-field of subsets of $\Omega$, and $\mu$ is a probability measure on $\m{F}$. Given two probability spaces $(\Omega, \m{F}, \mu)$ and $(\Omega', \m{F}', \mu')$, any measurable, measure-preserving map $T$ from $\Omega$ to $\Omega'$ will be referred to as a transformation of probability spaces, and we write $T: (\Omega, \m{F}, \mu) \to (\Omega', \m{F}', \mu')$. The product probability space $(\Omega\times \Omega', \m{F}\otimes \m{F}', \mu\times \mu')$ will often be denoted by $(\Omega, \m{F}, \mu) \times (\Omega', \m{F}', \mu')$.

Throughout this article, all equalities between measurable maps, sets, and functions are interpreted as strict pointwise equalities, rather than equivalences modulo sets of measure zero. This contrasts with Tsirelson's conventions \cite{Tsi04}, which rely on almost-everywhere equality. Our choice is motivated by consistency with classical projective-limit constructions \cite{Bochner} and with set-theoretic foundations.

Moreover, while Tsirelson's framework assumes that all measurable spaces are standard Borel spaces, we drop this assumption entirely in order to maintain the highest level of generality. This approach aligns with the development of an ``uncountable'' measure theory, as discussed in \cite{Tao}. In a similar spirit, we consider convolution systems over arbitrary linearly ordered sets rather than only over $\mathbb{R}$, as suggested in \cite{Tsi04}, and we do not assume the separability of Hilbert spaces, a condition crucial in Arveson's theory of product systems but one we avoid for broader applicability.

\section{Background}
\subsection{Convolution systems of probability spaces and subproduct systems of Hilbert spaces}
In this subsection, we present the main objects of study in this article, all of which were formally introduced by B. Tsirelson in \cite{Tsi03, Tsi04}, within the special setting discussed above.

\begin{definition}\label{kolmdef}(i) A convolution system $\cali{S}=\{\Omega_{s,t}, \m{F}_{s,t}, \mu_{s,t}, T_{r,s,t}\}_\mathbb{S}$ over a linearly ordered set $(\mathbb{S},\leq)$ consists of a family $\{(\Omega_{s,t}, \m{F}_{s,t}, \mu_{s,t})\,|\,s,\,t\in \mathbb{S},\,s<t\}$ of probability spaces and a family $\{T_{r,s,t}\,|\,r,\,s,\,t\in \mathbb{S},\,r<s<t\}$ of transformations $T_{r,s,t}: (\Omega_{r,s}, \m{F}_{r,s}, \mu_{r,s})\times (\Omega_{s,t},  \m{F}_{s,t},\mu_{s,t})\to (\Omega_{r,t}, \m{F}_{r,t}, \mu_{r,t})$, %$T_{r,s,t}: (\Omega_{r,s}\times \Omega_{s,t}, \m{F}_{r,s}\otimes \m{F} _{s,t}, \mu_{r,s}\times \mu_{s,t})\to (\Omega_{r,t}, \m{F}_{r,t}, \mu_{r,t})$
 called the multiplication of the system, which is associative in the sense that the diagram \begin{equation}\label{kass}
\begin{tikzcd}[column sep=60pt]
\Omega_{r,s}\times \Omega_{s,t}\times \Omega_{t,u} \arrow{r}{T_{r,s,t}\times \operatorname{id}_{\Omega_{t,u}}} \arrow[swap]{d}{\operatorname{id}_{\Omega_{r,s}}\times T_{s,t,u}}  & \Omega_{r,t}\times \Omega_{t,u} \arrow{d}{T_{r,t,u}} \\%
\Omega_{r,s}\times \Omega_{s,u} \arrow{r}{T_{r,s,u}}& \Omega_{r,u},
\end{tikzcd}\end{equation} 
 commutes 
 for all $r<s<t<u$ in $\mathbb{S}$. 
 
(ii) A continuous product of probability spaces, abbreviated as CPPS, is a  convolution system whose multiplication consists of bijective transformations of probability spaces.

(iii) A morphism of convolution systems, from $\cali{S}=\{\Omega_{s,t}, \m{F}_{s,t}, \mu_{s,t}, T_{r,s,t}\}_{\mathbb{S}}$ to  $\cali{S'}=\{\Omega_{s,t}', \m{F}_{s,t}', \mu_{s,t}', T_{r,s,t}'\}_{\mathbb{S}}$, is a family $\theta=\{\theta_{s,t}\}_{s<t}$ of transformations of probability spaces $\theta_{s,t}: (\Omega_{s,t}, \m{F}_{s,t}, \mu_{s,t})\to (\Omega_{s,t}', \m{F}_{s,t}', \mu_{s,t}')$, which makes the diagram
\begin{equation}\label{army}
\begin{tikzcd}[column sep=60pt]
\Omega_{r,s}\times \Omega_{s,t} \arrow{r}{\theta_{r,s}\times \theta_{s,t}} \arrow[swap]{d} {T_{r,s,t}}  & \Omega'_{r,s}\times \Omega'_{s,t} \arrow{d}{T_{r,s,t}'} \\%
\Omega_{r,t} \arrow{r}{\theta_{r,t}}& \Omega_{r,t}'
\end{tikzcd}\end{equation} 
commutative, for all $r<s<t$ in $\mathbb{S}$.

(iv) A flow system $((\Omega^*, \m{F}^*, P^*), \{X_{s,t}\})_{\cali{S}}$ over a convolution system $\cali{S}=\{\Omega_{s,t}, \m{F}_{s,t}, \mu_{s,t}, T_{r,s,t}\}_{\mathbb{S}}$ consists of a probability space $(\Omega^*, \m{F}^*, P^*)$ and a family $\{X_{s,t}\,|\,s,\,t\in\mathbb{S},\,s<t\}$ of transformations of probability spaces $X_{s,t}:(\Omega^*, \m{F}^*, P^*)\to (\Omega_{s,t}, \m{F}_{s,t}, \mu_{s,t})$ that satisfies the following conditions:
\begin{enumerate}
\item $\m{F}^*$ is the joint of the $\sigma$-fields $\m{F}_{s,t}^*=X_{s,t}^{-1}(\m{F}_{s,t})$, i.e., the $\sigma$-field generated by the field $\m{F}_0=\bigcup_{s<t}\m{F}_{s,t}^*$;
\item $X_{t_1,t_2}$,  $X_{t_2,t_3}$,....,  $X_{t_{n-1},t_n}$ are independent, for all $t_1<t_2<\cdots<t_n$ in $\mathbb{S}$;
\item $X_{r,t}(\omega)=T_{r,s,t}(X_{r,s}(\omega), X_{s,t}(\omega))$, $\omega\in \Omega^*$,  for all $r<s<t$ in $\mathbb{S}$.
\end{enumerate}
\end{definition} 

\begin{example}\label{exa1}
Let $(S, \mathcal{G})$ be a measurable semigroup, consisting of a semigroup $S$ and a $\sigma$-field $\mathcal{G}$ on $S$, such that the semigroup operation $(s,t) \mapsto st$ is $\mathcal{G}$-measurable. The convolution $\mu \ast \nu$ of two probability measures $\mu$ and $\nu$ on $\mathcal{G}$ is defined as the push-forward of the product measure $\mu \times \nu$ by the semigroup operation.

Suppose that $\mu$ is an idempotent probability measure on $\mathcal{G}$, i.e., $\mu \ast \mu = \mu$. Then it gives rise to a trivial convolution system $$\cali{S}_\mu = \{\Omega_{s,t}=S, \mathcal{F}_{s,t}= \mathcal{G}, \mu_{s,t} = \mu, T_{r,s,t}\}_\mathbb{S}$$ over any linearly ordered set $(\mathbb{S}, \leq)$, where $T_{r,s,t}$ is given by the semigroup operation for all $r < s < t$ in \(\mathbb{S}$.

We note that the structure of idempotent regular Borel probability measures on locally compact, Hausdorff, second-countable topological semigroups has been thoroughly studied (see, e.g., \cite[Th. 2.8]{HM}).
\end{example}

\begin{example}\label{exa2}
Let $\{\mu_t\}_{t \in \mathbb{R}}$ be a one-parameter convolution semigroup of probability measures over a measurable semigroup $(S, \mathcal{G})$. Then $$\cali{S}_{\{\mu_t\}} = \{\Omega_{s,t}, \mathcal{F}_{s,t}, \mu_{s,t}, T_{r,s,t}\}_\mathbb{R}$$ is a convolution system over the set of all real numbers, where $\Omega_{s,t} = S$, $\mathcal{F}_{s,t} = \mathcal{G}$, $\mu_{s,t} = \mu_{t-s}$, and $T_{r,s,t}$ is given by the semigroup operation for all real numbers $r < s < t$.
\end{example}

\begin{example}\label{exa3}
Let $(\mathbb{S}, \leq)$ be a linearly ordered set, equipped with the $\sigma$-field generated by the order topology. Let $(X, \mathcal{A})$ be a measurable space, and let $M = \{M_\omega\}_{\omega \in \Omega^*}$ be a random measure on $\mathbb{S} \times X$ with independent increments, defined on some probability space $(\Omega^*, \mathcal{F}^*, P^*)$.

For each $s, t \in \mathbb{S}$, with $s < t$, define
$$
\Omega_{s,t} = \{M_\omega\restriction_{(s,t] \times X} \mid \omega \in \Omega^*\},
$$
the collection of all random measure restrictions to the slice $(s,t] \times X$. Let $\mathcal{F}_{s,t}$ be the coordinate $\sigma$-field on $\Omega_{s,t}$, i.e., the smallest $\sigma$-field making all evaluation maps $\Omega_{s,t} \ni \mu \mapsto \mu(B)$ measurable for every measurable set $B \subseteq (s,t] \times X$. Let $\mu_{s,t}$ be the push-forward measure of $P^*$ under the restriction map $X_{s,t} : \Omega^* \to \Omega_{s,t}$, $
X_{s,t}(\omega) = M_\omega\restriction_{(s,t] \times X}.$

Additionally, we assume that $M$ has consistent increments, in the sense that for every $r < s < t$ in $\mathbb{S}$, and for every $\mu_1 \in \Omega_{r,s}$, $\mu_2 \in \Omega_{s,t}$, there exists an outcome $\omega \in \Omega^*$ such that
$$
M_\omega\restriction_{(r,s] \times X} = \mu_1
\quad \text{and} \quad
M_\omega\restriction_{(s,t] \times X} = \mu_2.
$$
Under this assumption, the concatenation map $T_{r,s,t} : \Omega_{r,s} \times \Omega_{s,t} \to \Omega_{r,t},$
defined by
$$
(T_{r,s,t}(\mu_1, \mu_2))(B) = \mu_1(B \cap ((r,s] \times X)) + \mu_2(B \cap ((s,t] \times X)),
$$
for every measurable subset $B \subseteq (r,t] \times X$, is well-defined, and the system $$\cali{S}_{\{M_\omega\}}=\{\Omega_{s,t}, \mathcal{F}_{s,t}, \mu_{s,t}, T_{r,s,t}\}_\mathbb{S}$$ is a convolution system. Moreover, by replacing $\mathcal{F}^*$, if needed, with the $\sigma$-field generated by $\bigcup_{s < t} X_{s,t}^{-1}(\mathcal{F}_{s,t})$, the system $((\Omega^*, \mathcal{F}^*, P^*), \{X_{s,t}\})$ becomes a flow system over the convolution system $\cali{S}_{\{M_\omega\}}$, as can be readily verified.

This construction applies to many classical examples, such as Poisson random measures, compensated jump processes, and other random measures with independent increments.
\end{example}

\begin{observation}\label{obsa1}
Tsirelson's approach to constructing CPPSs from convolution systems of standard probability spaces \cite[Section 3c]{Tsi04} does not directly use the convolution system itself. Instead, it relies on the existence of an associated ``almost-everywhere'' (a.e.) flow system. In \cite{Tsi04}, Tsirelson further established that a convolution system over $\mathbb{S} = (\mathbb{R}, \leq)$ admits an a.e. flow system on a standard probability space if and only if it is separable, as defined in Definition 3b4 of the same reference.

For instance, consider the group $S = (\mathbb{Z}/m\mathbb{Z}, +)$, equipped with the uniform distribution $\mu_m$. Then the convolution system $\cali{S}_{\mu_m}$, constructed as in Example \ref{exa1}, is not separable. In contrast, the system described in Example \ref{exa2}, based on $S = (\mathbb{R}, +)$, is separable.
\end{observation}

Convolution systems and CPPSs can be regarded as the probabilistic counterparts to the notions of two-parameter subproduct systems and product systems of Hilbert spaces, respectively. However, the former appear to be more versatile than the latter. The concept of two-parameter product systems of Hilbert spaces was first introduced by Tsirelson in \cite{Tsi03} under the name ``local continuous product of Hilbert spaces,'' obtained by excluding unbounded intervals from the definition of a continuous product of Hilbert spaces. Meanwhile, two-parameter subproduct systems were formally introduced in \cite{FK} under the name ``Tsirelson subproduct system,'' although the concept was well known long before their formal introduction.

\begin{definition}
A two-parameter subproduct system of Hilbert spaces, or subproduct system for short, $\cali{H}=\{H_{s,t}, U_{r,s,t}\}_\mathbb{S}$ over a linearly ordered set $(\mathbb{S},\leq)$ consists of a family  $\{H_{s,t}\,|\, s,\,t\in\mathbb{S},\, s<t\}$ of Hilbert spaces $H_{s,t}$ and a family $\{U_{r,s,t}\,|\,r,\,s,\,t\in\mathbb{S},\,r<s<t\}$ 
 of isometric operators $U_{r,s,t}: H_{r,t} \to H_{r,s}\otimes H_{s,t}$ that satisfy the co-associativity law 
 \begin{eqnarray}\label{Jan24cc}
 \left(1_{H_{r,s}}\otimes U_{s,t,u}\right)U_{r,s,u}=\left(U_{r,s,t}\otimes 1_{H_{t,u}}  \right)U_{r,t,u},
 \end{eqnarray} 
for all $r < s < t < u$ in $\mathbb{S}$.

If the operators $U_{r,s,t}$ are all unitary operators, then $\cali{H}=\{H_{s,t}, U_{r,s,t}\}_\mathbb{S}$ will be referred to as a two-parameter product system of Hilbert spaces, or product system for short.

The concept of morphism or isomorphism of subproduct systems is naturally defined, similarly to that of morphism or isomorphism of convolution systems.
\end{definition}

\begin{observation}\label{obsa2} Any convolution system $\cali{S}=\{\Omega_{s,t}, \m{F}_{s,t}, \mu_{s,t}, T_{r,s,t}\}_{\mathbb{S}}$  gives rise to the subproduct system $$L^2(\cali{S}):=\{L^2( \mu_{s,t}),\,U_{T_{r,s,t}}\}_{\mathbb{S}},$$ where $U_{T_{r,s,t}}$ is the Koopman isometry induced by $T_{r,s,t}$, i.e. $U_{T_{r,s,t}}(f) = f\comp T_{r,s,t}$ for $f\in L^2(\mu_{r,t})$. If $\cali{S}$ is a CPPS, then $L^2(\cali{S})$ is a product system. Moreover, any morphism $\theta=\{\theta_{s,t}\}_{0<s<t}$ of convolution systems, from $\cali{S}$ to $\cali{S}'$, induces the morphism $U_\theta=\{U_{\theta_{s,t}}\}_{0<s<t}$ of subproduct systems, from $L^2(\cali{S}')$ to $L^2(\cali{S})$.
 \end{observation}

\subsection{Projective systems of probability spaces} In this subsection, we briefly recall the concept of a projective system and its projective limit in the sense of Bochner \cite{Bochner}, as this framework plays a central role in the present paper. For a more detailed discussion of the topic, we refer the reader to \cite{Bochner, Rao 71, Chol}.

A projective system of probability spaces over a directed set $(I, \leq)$ is a pair $\{(\Omega_i, \m{F}_i, \mu_i)\}, T_{i,j}\}$ consisting of a family  $\{(\Omega_i, \m{F}_i, \mu_i)\}_{i\in I}$ of probability spaces and a family  $\{T_{i,j}\}_{i,\,j\in I,\,i\leq j}$ of transformations of probability spaces $T_{i,j}:(\Omega_j , \m{F}_j , \mu_j) \to (\Omega_i, \m{F}_i, \mu_i)$, $i\leq j$, satisfying the compatibility condition  $T_{i,j}T_{j,k}=T_{i,k}$, for all $i\leq j\leq k$, and $T_{i,i}=\op{id}_{\Omega_i}$.

Let $\Omega= \limproj\{\Omega_i, T_{i,j}\}$ be the projective limit of the projective system of sets $\{\Omega_i, T_{i,j}\}$ over $(I, \leq)$, i.e., $\Omega$ is the subset of the cartesian product $\Omega^I = \bigtimes _{i\in I}\Omega_i$ consisting of those elements
$\omega= \{\omega_i\}_{i\in I}$, called threads, such that for each $i\leq j$ in $I$, $\omega_i = T_{i,j}(\omega_j)$. It is worth noting that the set $\Omega$ may be empty, even when the maps $T_{i,j}$ are surjective \cite{Henkin, Waterhouse}. If $\Omega$ is non-empty and the coordinate projections $$T_i:\Omega\ni  \{\omega_j\}_{j\in I}\mapsto\omega_i\in \Omega_i$$ are all surjective, then, using terminology from \cite{Bochner}, we say that the projective system of sets $\{\Omega_i, T_{i,j}\}$ is {\em simply maximal}. In this case, the maps $T_{i,j}$ must also be all surjective.

Assuming further that the system  $\{\Omega_i, T_{i,j}\}$ is simply maximal, we consider the field $\m{F}_0=\bigcup_{i\in I} T_i^{-1}(\m{F}_i)$ and the joint
$\m{F} = \bigvee_{i\in I} T_i^{-1}(\m{F}_i)$ of the $\sigma$-fields $T_i^{-1}(\m{F}_i)$.
Let $\mu :\m{F}_0\to [0,1]$ be the finitely
additive set function, defined uniquely by the equation \begin{eqnarray}\label{crr}\mu(T_i^{-1}(A)) =\mu_i(A),\;A\in \m{F}_i.\end{eqnarray}
Generally, $\mu$ may fail to be $\sigma$-additive and thus cannot be extended to $\mathcal{F}$.
\begin{definition}
A simply maximal projective system of probability spaces $\{(\Omega_i, \m{F}_i, \mu_i), T_{i,j}\}$ is said to be convergent if $\mu$ is $\sigma$-additive on $\m{F}_0$.\end{definition}  If this is the case,  then we shall use the same notation, $\mu$, to denote the 
(unique) $\sigma$-additive extension of $\mu$ to $\m{F}$. The resulting probability space $$(\Omega,\m{F}, \mu):=\limproj_{I}\{(\Omega_i, \m{F}_i, \mu_i), T_{i,j}\}$$ is called
the projective limit of $\{(\Omega_i , \m{F}_i, \mu_i), T_{i,j}\}$. 

Numerous criteria for the existence of a projective limit have been identified (see \cite{Rao71, Chol, Frol} and the references therein). In this paper, we employ a straightforward yet effective approach that carries a categorical flavour: if $\{(\Omega_i, \m{F}_i, \mu_i)\}, T_{i,j}\}$ admits a lower bound, consisting of a probability space $(N, \m{N}, \nu)$ and a family $\{S_i\}_{i\in I}$ of surjective transformations $S_i: (N, \m{N}, \nu)\to (\Omega_i, \m{F}_i, \mu_i)$ such that $\m{N}=\bigvee_{i\in I}S_i^{-1}(\m{F}_i)$ and $T_{i,j}S_j=S_i$, for all $i,\,j\in I$ with $i\leq j$, then the set-function $\mu$ in (\ref{crr}) must be $\sigma$-additive, and thus the projective limit exists. From this perspective, the projective limit $(\Omega,\m{F}, \mu)$ is a greatest lower bound in the sense that any other lower bound factors through it via a surjective transformation. Consequently, it is unique up to a bijective transformation of such greatest lower bounds, and also satisfies $(\Omega,\m{F}, \mu)=\limproj_{I_0}\{(\Omega_i, \m{F}_i, \mu_i), T_{i,j}\}$, for every cofinal subset $I_0$ of $(I, \leq)$.

\subsection{Finite Partitions}
Let $(\mathbb{S},\leq)$ be a linearly ordered set. For any two elements $s,\,t\in \mathbb{S}$, $s < t$, consider the partially ordered set $(\m{K}_{s,t},\subseteq)$ of all finite partitions of the interval $[s,t]$, ordered by inclusion. We notice that if $I=\{s=\iota_0<\iota_1<\iota_2<\,\dots<\iota_m<\iota_{m+1}=t\}\in\m{K}_{s,t}$, then any refinement $J\in \m{K}_{s,t}$ of $I$, i.e., $I\subseteq J$, can be written as \begin{eqnarray}\label{decoopa}J=I_0\cup I_1\cup \cdots \cup I_m,\end{eqnarray} where $I_i=\{j\in J, \iota_i\leq j\leq \iota_{i+1}\}=\{\iota_i=\iota_{i_0}<\iota_{i_1}<\dots<\iota_{i_{n_{I_i}}}<\iota_{i+1}\}\in \m{K}_{\iota_i,\iota_{i+1}},$ for some $n_{I_i}\in\mathbb{N}$ depending on the partition $I_i$, for all $0\leq i\leq m$. 
 
 We also consider the set $\m{K}=\bigcup _{s<t}\m{K}_{s,t}$ of all finite subsets $I\subset \mathbb{S}$, $|I|\geq 2$, ordered by inclusion as well. As above, if $I\in\m{K}_{s,t}$ and  $J\in \m{K}$ is a refinement of $I$, then $J$ can be written as 

\begin{eqnarray}\label{abba} J={I}_{L}\cup \pe{I}\cup {I}_R,\end{eqnarray} where ${I}_L=\{j\in J\mid j\leq \iota_0\}$, ${I}_R=\{j\in J\mid j\geq \iota_{m+1}\}$ and $\pe{I} = I_0\cup \cdots \cup I_m\in\m{K}_{s,t}$, with $I_k$ defined as above. Note that $\pe{I}$ is a refinement of $I$.

\section{Convergent convolution systems and associated CPPSs}\label{ch.1} 

Let $\cali{S}=\{\Omega_{s,t}, \m{F}_{s,t}, \mu_{s,t}, T_{r,s,t}\}_\mathbb{S}$ be a convolution system.
 For any two elements $s,\,t\in \mathbb{S}$,  $s < t$, and any partition $I\in \m{K}_{s,t}$, $I=\{s=\iota_0<\iota_1<\iota_2<\,\dots<\iota_m<\iota_{m+1}=t\},$ we consider the product measure space $(\Omega_I, \m{F}_I, \mu_I)$ over $I$, where $\Omega_I=\Omega_{\iota_0, \iota_1}\times \Omega_{\iota_1,\iota_2}\times \dots\times \Omega_{\iota_m,\iota_{m+1}}$, $\m{F}_I=\m{F}_{\iota_0, \iota_1}\otimes \m{F}_{\iota_1,\iota_2}\otimes \dots\otimes \m{F}_{\iota_m,\iota_{m+1}},$ and $\mu_I=\mu_{\iota_0, \iota_1}\times \mu_{\iota_1,\iota_2}\times \dots\times \mu_{\iota_m,\iota_{m+1}}.$
 
For any refinement $J\in \m{K}_{s,t}$ of $I$, we also consider the mapping $T_{I,J}:\Omega_J\to \Omega_I$ defined as follows: 
 
 (i) If $I$ is the trivial partition, $I=\{s,t\}$, and $J=\{s=j_0<j_1<j_2<\,\dots<j_n<j_{n+1}=t\},$ then $T_{I, J}: \Omega_J \to \Omega_{s,t}$ is defined iteratively as
\begin{eqnarray*}\label{hooop}T_{I, J}=\begin{cases} T_{j_0,j_1, j_2}&\,\mbox{if}\, n=1\\ T_{j_{0},j_{n},j_{ n+1}}\comp(T_{\{j_0,j_n\},J\setminus \{j_{n+1}\}}\times \operatorname{id}_{\Omega_{j_n,j_{n+1}}})&\,\mbox{if}\, n\geq 2\end{cases}\end{eqnarray*} where
$J\setminus\{j_{n+1}\}\in \m{K}_{s,j_n}$ is the partition obtained by removing the endpoint $t=j_{n+1}$ from $J$.

(ii) If $I=\{s=\iota_0<\iota_1<\iota_2<\,\dots<\iota_m<\iota_{m+1}=t\},$ $I\subseteq J$, are arbitrary partitions, then by writing $J$ it as in (\ref{decoopa}), we define  \begin{eqnarray*}\label{jan16}T_{I,J}=T_{\{\iota_0,\iota_1\}, I_0}\times T_{\{\iota_1,\iota_2\}, I_1}\times \dots T_{\{\iota_m,\iota_{m+1}\}, I_m},\end{eqnarray*}
where $T_{\{\iota_k,\iota_{k+1}\}, I_k}:\Omega_{I_k}\to \Omega_{\iota_k,\iota_{k+1}}$ are as in (i).

(iii) If $I=J$, we set $T_{I,I}=\op{id}_{\Omega_I}$.

The construction carried out above ensures that $T_{I,J}:(\Omega_J, \m{F}_J, \mu_J)\to (\Omega_I, \m{F}_I, \mu_I )$ is a transformation of probability spaces, for all $I,\,J\in \m{K}_{s,t}$, $I\subseteq J$. The newly created system is projective, as shown below.

\begin{proposition}\label{projsysPst}
Let $\cali{S}=\{\Omega_{s,t}, \m{F}_{s,t}, \mu_{s,t}, T_{r,s,t}\}_\mathbb{S}$ be a convolution system. Then the system $(\{(\Omega_I, \m{F}_I, \mu_I)\}, \{T_{I,J}\})$ is a projective system of probability spaces over the directed set $(\mathcal{K}_{s,t},\subseteq)$, for all $s<t$ in $\mathbb{S}$.
\end{proposition}
\begin{proof} 
Let $s<t$ be two elements of $\mathbb{S}$. We show that $T_{I,K}=T_{I,J}T_{J,K}$, for all partitions $I,\,J,\,K\in \mathcal{K}_{s,t}$, $I\subseteq J\subseteq K$. %The verification of this compatibility identity can be conducted in a manner analogous to that in \cite[Proposition 3.2]{FK}. We present it here briefly for the sake of completeness. 
If $I=\{s,t\}$ is the trivial partition, then the identity can be deduced easily. For a general partition $I=\{s=\iota_0<\iota_1<\cdots <\iota_m<\iota_{m+1}=t\}$, we use  (\ref{decoopa}) to decompose $J$ in terms of $I$ as $J=I_0\cup I_1\cup \cdots \cup I_m$, where $I_i\in  \m{K}_{\iota_i,\iota_{i+1}}$, for all $0\leq i\leq m$. %$I_i=\{\iota_i=\iota_{i_0}<\iota_{i_1}<\cdots <\iota_{i_{n_{I_i}}}<\iota_{i+1}\}$, for some $n_{I_i}\in\mathbb{N}$. 
Similarly, we can decompose $K$ in terms of $J$, thus $K=J_0\cup J_1\cup \cdots \cup J_\ell$, where $J_0 \in \m{K}_{\iota_{0_0},\iota_{0_1}}$, $J_1\in \mathcal{K}_{\iota_{0_1},\iota_{0_2}}$, $\cdots$, $J_\ell \in \mathcal{K}_{\iota_{m_ {n_{I_m}}}, \iota_{m+1}}$. However, we can also decompose $K$ in terms of $I$. Explicitly, we get $K=I'_0\cup I'_1\cup \cdots \cup I'_m$, where $I'_0= \bigcup\limits_{i=0}^{n_{I_0}} J_i$ and $I'_i= \bigcup\limits_{j=n_{I_{i-1}+1}}^{n_{I_i}} J_j$ for all $1\leq i \leq m$. Via this new decomposition, we have that $T_{I_0,I'_0}= T_{\{\iota_{0_0},\iota_{0_1}\}, J_0}\times T_{\{\iota_{0_1}, \iota_{0_2}\}, J_1}\times \cdots \times T_{\{\iota_{0_{n_{I_0}}}, \iota_{1}\}, J_{n_{I_0}}}$, and  $    T_{I_k,I'_k} = T_{\{\iota_{k_0}, \iota_{k_1}\},  J_{n_{I_{k-1}}+1}} \times T_{\{\iota_{k_1}, \iota_{k_2}\}, J_{n_{I_{k-1}}+2}}\times \cdots \times T_{\{\iota_{k_{n_{I_k}}}, \iota_{k+1}\}, J_{n_{I_k}}}$, for all $1\leq k\leq m$, so $T_{J,K}=\bigtimes_{k=0}^mT_{I_k, I_k'}$. Consequently, $$
   T_{I,K} =\bigtimes_{k=0}^m T_{\{\iota_k,\iota_{k+1}\}, I_k'}=\bigtimes_{k=0}^m T_{\{\iota_k,\iota_{k+1}\},I_k}T_{I_k,I_k'}
       =T_{I,J} T_{J,K}$$ as required.
\end{proof}

\begin{definition} A convolution system $\cali{S}=\{\Omega_{s,t}, \m{F}_{s,t}, \mu_{s,t}, T_{r,s,t}\}_\mathbb{S}$ is said to be convergent if for any two elements $s<t$ of $\mathbb{S}$, the projective system  $(\{(\Omega_I, \m{F}_I, \mu_I)\}, \{T_{I,J}\})$ over  $(\mathcal{K}_{s,t},\subseteq)$ is convergent.\end{definition} We denote by $(\Omega^\flat_{s,t}, \m{F}_{s,t}^\flat, \mu_{s,t}^\flat)$ the projective limit: $$(\Omega^\flat_{s,t}, \m{F}_{s,t}^\flat, \mu_{s,t}^\flat)=\varprojlim_{\mathcal{K}_{s,t}} \{(\Omega_I, \m{F}_I, \mu_I), T_{I,J}\},
$$ of the convergent projective system $\{(\Omega_I, \m{F}_I, \mu_I), T_{I,J}\}$ over $(\mathcal{K}_{s,t},\subseteq)$. We also consider the coordinate projections $T^\flat_I:  (\Omega^\flat_{s,t}, \m{F}_{s,t}^\flat, \mu_{s,t}^\flat)\to (\Omega_I, \m{F}_I, \mu_I)$, for all $I\in\m{K}_{s,t}$. These transformations satisfy the compatibility relation $T_I^\flat=T_{I,J}T_J^\flat$, for all $I,\,J\in \m{K}_{s,t}$, $I\subseteq J$.

The following result enables the construction of a CPPS from a convergent convolution system.
\begin{theorem}\label{projequiv}
 Let $\cali{S}=\{\Omega_{s,t}, \m{F}_{s,t}, \mu_{s,t}, T_{r,s,t}\}_\mathbb{S}$ be a convergent convolution system. For any $r<s<t$ in $\mathbb{S}$, there exists a bijective transformation of probability spaces  $$T^\flat_{r,s,t}: (\Omega^\flat_{r,s}, \m{F}_{r,s}^\flat, \mu_{r,s}^\flat)\times (\Omega^\flat_{s,t}, \m{F}_{s,t}^\flat,  \mu_{s,t}^\flat)\to (\Omega^\flat_{r,t}, \m{F}_{r,t}^\flat, \mu_{r,t}^\flat)$$
  that makes the system $\cali{S}^\flat= \{\Omega_{s,t}^\flat, \m{F}_{s,t}^\flat, \mu_{s,t}^\flat, T_{r,s,t}^\flat\}_\mathbb{S}$ into a CPPS. 
 \end{theorem}
 \begin{proof}
Let $r,\, t\in\mathbb{S}$ with $r<t$. For any $s\in\mathbb{S}$ satisfying $r<s<t$, let $\m{K}_{r,s,t}$ be the set of all partitions $I\in \m{K}_{r,t}$ such that $s\in I$. Since $\m{K}_{r,s,t}$ is a cofinal subset of $(\mathcal{K}_{r,t},\subseteq)$, we have that $(\Omega^\flat_{s,t}, \m{F}_{s,t}^\flat, \mu_{s,t}^\flat)$ is the projective limit of the convergent projective system $\{(\Omega_I, \m{F}_I, \mu_I), T_{I,J}\}$ over $(\mathcal{K}_{r, s,t}, \subseteq)$. 
We freely identify the poset $(\m{K}_{r, s,t}, \subseteq)$ with $\m{K}_{r,s}\times \m{K}_{s,t}$, considered with the product order, via the order isomorphism $\m{K}_{r, s,t}\ni I=I_s\cup {}_sI\mapsto I_s\times {}_sI\in \m{K}_{r,s}\times \m{K}_{s,t}$. Given that $(\Omega_I, \m{F}_I, \mu_I)=(\Omega_{I_s}, \m{F}_{I_s}, \mu_{I_s})\times (\Omega_{{}_sI}, \m{F}_{{}_sI}, \mu_{{}_sI})$, for all $ I\in \m{K}_{r,s,t}$, and $T_{I,J}=T_{I_s, J_s}\times T_{{}_sI, {}_sJ}$, for all $I\subseteq J$ in  $\m{K}_{r,t}$, and noting that the product of two convergent projective systems is again convergent, with the projective limit given by the product of the individual limits, we conclude that the product probability space $(\Omega^\flat_{r,s}, \m{F}_{r,s}^\flat, \mu_{r,s}^\flat)\times (\Omega^\flat_{s,t}, \m{F}_{s,t}^\flat, \mu_{s,t}^\flat)$ is a greatest lower bound of the convergent projective system $\{(\Omega_I, \m{F}_I, \mu_I), T_{I,J}\}$ over  $(\mathcal{K}_{s,t},\subseteq)$. Consequently, there exists a bijective transformation of probability spaces $T^\flat_{r,s,t}: (\Omega^\flat_{r,s}, \m{F}_{r,s}^\flat, \mu_{r,s}^\flat)\times (\Omega^\flat_{s,t}, \m{F}_{s,t}^\flat, \mu_{s,t}^\flat)\to (\Omega^\flat_{r,t}, \m{F}_{r,t}^\flat, \mu_{r,t}^\flat)$ such that 
 \begin{eqnarray}\label{Thursday} T^\flat_{I}T^\flat_{r,s,t} = T^\flat_{I_s}\times T^\flat_{{}_sI}
 \end{eqnarray}
 for all $I=I_s\cup{}_sI\in \mathcal{K}_{r,s,t}$.
 
 It remains to show that the family $\{T_{r,s,t}^\flat\}$ thus constructed is associative. Let then $r<s<t<u$ be some elements of $\mathbb{S}$. Identifying, as above, the poset $(\m{K}_{r,s,t,u}, \subseteq )$ of all partitions $I\in \m{K}_{r,u}$ such that $s,\,t\in I$ with the product poset $\m{K}_{r,s}\times \m{K}_{s,t}\times \m{K}_{t,u}$ via $I=I_{s}\cup {}_sI_t\cup {}_tI\mapsto I_{s}\times {}_sI_t\times {}_tI$, we have
 
 \begin{eqnarray}\label{sf}
         T^\flat_{I}T^\flat_{r,s,u}(\operatorname{id_{\Omega_{r,s}^\flat}}\times T^\flat_{s,t,u}) 
          &=& (T^\flat_{I_s}\times T^\flat_{{}_sI_t\cup {}_tI})(\operatorname{id_{\Omega_{r,s}^\flat}}\times T^\flat_{s,t,u}) \\\nonumber&=&T^\flat_{I_s}\times T^\flat_{{}_sI_t}\times T^\flat_{I_t}.
 \end{eqnarray}
Similarly, $T_{I}^\flat T^\flat_{r,t,u}( T^\flat_{r,s,t}\times\operatorname{id}_{\Omega_{t,u}^\flat}) = T^\flat_{I_s}\times T^\flat_{{}_sI_t}\times T^\flat_{I_t}$, and the conclusion follows.
\end{proof}

\begin{definition} \label{Pprost}The system $\cali{S}^\flat= \{\Omega_{s,t}^\flat, \m{F}_{s,t}^\flat, \mu_{s,t}^\flat, T_{r,s,t}^\flat\}_\mathbb{S}$ will be referred to as the projective CPPS associated with $\cali{S}$.
\end{definition}

If $T^\flat_{s,t}:(\Omega^\flat_{s,t}, \m{F}_{s,t}^\flat, \mu_{s,t}^\flat)\to (\Omega_{s,t}, \m{F}_{s,t}, \mu_{s,t})$ is the coordinate projection corresponding to the trivial partition $I=\{s,t\}$, then the resulting family $\tau_\cali{S}=\{T^\flat_{s,t}\}_{s<t}$ is a morphism of convolution systems from $\cali{S}^\flat$ to $\cali{S}$. The following result follows immediately from Theorem \ref{projequiv} and from the properties of the projective limit.

\begin{corollary} If $\theta:\cali{S}_1\to\cali{S}_2$ is an isomorphism of convergent convolution systems, then there exists an isomorphism of CPPSs $\theta^\flat:\cali{S}_1^\flat\to\cali{S}_2^\flat$ such that $\theta\tau_{\cali{S}_1}=\tau_{\cali{S}_2}\theta^\flat$.
\end{corollary}

We conclude this section by noticing that the product system of Hilbert spaces $L^2(\cali{S}^\flat)$ of the projective CPPS $\cali{S}^\flat$ associated with a convergent convolution system $\cali{S}=\{\Omega_{s,t}, \m{F}_{s,t}, \mu_{s,t}, T_{r,s,t}\}_\mathbb{S}$ can be described in terms of the subproduct system of Hilbert spaces  $L^2(\cali{S})$ through an inductive limit construction. Concretely, for each $s<t$ in $\mathbb{S}$, let 
$$H_{s,t}=\limind_{I\in\m{K}_{s,t}}L^2(\mu_I)$$ be the inductive limit of the inductive system of Hilbert spaces $(L^2(\mu_I), U_{T_{I, J}})$ over $(\m{K}_{s,t}, \subseteq)$ with connecting isometries $V_I:L^2(\mu_I)\to H_{s,t}$, $I\in \m{K}_{s,t}$ satisfying the compatibility relation $V_JU_{T_{I,J}}=V_I$ for all $I\subseteq J$ in $\m{K}_{s,t}$, where \(U_{T_{I,J}}\) denotes the Koopman isometry induced by the transformation $T_{I,J}$. (For the classical construction of inductive limits of Hilbert spaces and further properties, see \cite[11.5.26]{KR}.) Arguing as in Theorem \ref{projequiv}, one can find a unique unitary operator $U_{r,s,t}:H_{r,t}\rightarrow H_{r,s}\otimes H_{s,t}$ that satisfies the identity $U_{r,s,t} V_{I}=V_{I_s}\otimes V_{{}_sI}$, 
for all $I=I_s\cup {}_sI\in\m{K}_{r,s,t}$. The system $$\cali{H}=\{H_{s,t}, U_{r,s,t}\}_\mathbb{S}$$ is a product system of Hilbert spaces. We then have:
\begin{theorem}\label{PS}If $\cali{S}=\{\Omega_{s,t}, \m{F}_{s,t}, \mu_{s,t}, T_{r,s,t}\}_\mathbb{S}$ is a convergent convolution system, then 
the product systems of Hilbert spaces $\cali{H}=\{H_{s,t}, U_{r,s,t}\}_\mathbb{S}$ and $L^2(\cali{S}^\flat)$ are isomorphic. 
\end{theorem}

\begin{proof}
Because $U_{T_J^\flat}U_{T_{I, J}}=U_{T_I^\flat}$, for all $I\subseteq J$, $I, J\in\m{K}_{s,t}$, and the set $\bigcup _{I\in \m{K}_{s,t}}U_{T_I^\flat}(L^2(\mu_I))$ is everywhere dense in $L^2(\mu_{s,t}^\flat)$, it follows from the universal property of the inductive limit that there exists a unique unitary operator $\theta_{s,t}:H_{s,t}\to L^2(\mu_{s,t}^\flat)$ such that $\theta_{s,t}V_I=U_{T^\flat_I}$, for all $I\in\m{K}_{s,t}$. Moreover, since \begin{eqnarray*}(\theta_{r,s}\otimes \theta_{s,t})U_{r,s,t}V_I&=&\theta_{r,s}V_{I_s}\otimes \theta_{s,t}V_{{}_sI}=U_{T^\flat_{I_s}}\otimes U_{T^\flat_{{}_sI}}\\&=&U_{T^\flat_{r,s,t}}U_{T^\flat_I}=U_{T^\flat_{r,s,t}}\theta_{r,t}V_I,\end{eqnarray*} for all $I=I_s\cup{}_sI\in\m{K}_{r,s,t}$, and $r<s<t$ in $\mathbb{S}$, we deduce that $(\theta_{r,s}\otimes \theta_{s,t})U_{r,s,t}=U_{T^\flat_{r,s,t}}\theta_{r,t}$, i.e., $\{\theta_{s,t}\}$ is an isomorphism of product systems.\end{proof}

 \section{$K$-convergent convolution systems and associated flow systems} 

Let $\cali{S}=\{\Omega_{s,t}, \m{F}_{s,t}, \mu_{s,t}, T_{r,s,t}\}_\mathbb{S}$ be a convolution system. Consider two partitions $I,\,J\in \m{K}$, $I\subseteq J$, and write $J={I}_{L}\cup \pe{I}\cup {I}_R$ as in (\ref{abba}). We set \begin{eqnarray*}\label{rocks}X_{I,J}= \begin{cases}T_{I,J}& \text{if } I,J\in \m{K}_{s,t},\,\text{for some}\, s<t\,\text{in}\;\mathbb{S}\\
(T_{I,\pe{I}})\comp (\pi_{\pe{I},J})& \text{otherwise} \end{cases}\end{eqnarray*} where $\pi_{\pe{I},J}:\Omega_J\to \Omega_{\pe{I}}$ is the coordinate projection. Then  $X_{I,J}:(\Omega_J, \m{F}_J, \mu_J)\to (\Omega_I, \m{F}_I, \mu_I)$ is a transformation of probability spaces. Moreover:

\begin{proposition}\label{projsysPst2}
The system $\{(\Omega_I, \m{F}_I, \mu_I), X_{I,J}\}$ is a projective system of probability spaces over the directed set $(\mathcal{K},\subseteq)$.
\end{proposition}
\begin{proof}The compatibility relation $X_{I,J}X_{J,K}=X_{I,K}$, for all $I,\,J,\,K \in\mathcal{K}$ satisfying $I\subseteq J\subseteq K$, is verified using the same approach as in Proposition \ref{projsysPst}. For this, suppose $I\in \mathcal{K}_{q,r}$, $J\in \mathcal{K}_{s,t}$, and $K\in \mathcal{K}_{u,v}$. %If $q=s$ and $t=t$, then the identity can be easily deduced. If, on the other hand,
where $u<s<q<r<t<v$. Following (\ref{abba}), we decompose $J$ with respect to $I$ as $J={I}_{L}\cup \pe{I}\cup {I}_R$, $K$ with respect to $J$ as  $K={J}_{L}\cup \pe{J}\cup {J}_R$, and $K$ with respect to $I$ as $K={I}_{L}'\cup \pe{I'}\cup {I}_R'.$ Note that the partitions $I,\pe{I}, \pe{I'}\in \m{K}_{q,r}$ satisfy $I\subseteq \pe{I}\subseteq \pe{I'}$, and the partitions $J,\pe{J}\in \m{K}_{s,t}$ satisfy $J\subseteq \pe{J}$. It then follows that $\pe{I'}\subseteq \pe{J},$ and decomposition (\ref{abba}) of $\pe{J}$ with respect to $\pe{I'}$ is given by $\pe{J}=(\pe{J}\cap [u,s])\cup \pe{I'}\cup(\pe{J}\cap [t,v])$. Using this, we have
\begin{eqnarray*}\pi_{\pe{I},J}\comp (T_{J,\pe{J}})&=&\pi_{\pe{I},J}\comp (T_{J\cap [u,s],\pe{J}\cap [u,s]}\times T_{J\cap [s,t],\pe{J}\cap [s,t]}\times T_{J\cap [t,v],\pe{J}\cap [t,v]})\\&=&\pi_{\pe{I},J}\comp (T_{I_L,\pe{J}\cap [u,s]}\times T_{\pe{I},\pe{I'}}\times T_{I_R,\pe{J}\cap [t,v]})\\&=&(T_{\pe{I}, \pe{I'}})\comp \pi_{\pe{I'},\pe{J}}.\end{eqnarray*}Therefore \begin{eqnarray*}X_{I,J}X_{J,K}&=&(T_{I,\pe{I}})\comp(\pi_{\pe{I}, J})\comp (T_{J, \pe{J}})\comp (\pi_{\pe{J}, K})=(T_{I,\pe{I}})\comp (T_{\pe{I},\pe{I'}})\comp (\pi_{\pe{I'}, \pe{J}})\comp (\pi_{\pe{J}, K})\\&=&X_{I, K},\end{eqnarray*}as needed.
\end{proof}
\begin{definition} A convolution system $\cali{S}=\{\Omega_{s,t}, \m{F}_{s,t}, \mu_{s,t}, T_{r,s,t}\}_\mathbb{S}$ is said to be $K$-convergent if the projective system $\{(\Omega_I, \m{F}_I, \mu_I), X_{I,J}\}$ over  $(\mathcal{K},\subseteq)$ is convergent.\end{definition}  If $\cali{S}=\{\Omega_{s,t}, \m{F}_{s,t}, \mu_{s,t}, T_{r,s,t}\}_\mathbb{S}$ is $K$-convergent, then we denote by $(\Omega^\flat, \m{F}^\flat, P^\flat)$ the projective limit of this projective system, $$(\Omega^\flat, \m{F}^\flat, P^\flat)=\varprojlim_{\mathcal{K}} \{(\Omega_I, \m{F}_I, \mu_I), X_{I,J}\},
$$ and by $X_I:  (\Omega^\flat, \m{F}^\flat, P^\flat)\to (\Omega_I, \m{F}_I, \mu_I)$ the coordinate projections; they satisfy the compatibility relation $X_I=X_{I,J}X_J$, for all $I,\,J\in \m{K}$, $I\subseteq J$. For a trivial partition $I=\{s< t\}$, we simply write $X_{s,t}$ instead of $X_{\{s, t\}}$.

The following proposition connects the main concepts of this paper.
\begin{proposition}\label{Kimp} A convolution system is $K$-convergent if and only if it admits a flow system. Moreover,  
any $K$-convergent convolution system is convergent.
\end{proposition}
\begin{proof}
Let $\cali{S}=\{\Omega_{s,t}, \m{F}_{s,t}, \mu_{s,t}, T_{r,s,t}\}_\mathbb{S}$ be a $K$-convergent convolution system. Then one can readily see that $$((\Omega^\flat, \m{F}^\flat, P^\flat), \{X_{s,t}\})_{\cali{S}}$$
 is a flow system over $\cali{S}$. In addition, the projective system of sets $\{\Omega_I,  T_{I,J}\}$ over $(\mathcal{K}_{s,t},\subseteq)$ is simply maximal, for any two elements  $s<t$  in $\mathbb{S}$. Furthermore, if $\pi_{s,t}:\Omega^\flat \to \Omega_{s,t}^\flat$ denotes the coordinate projection onto $\Omega^\flat$, then $T^\flat_I\pi_{s,t}=X_I$, for all $I\in \m{K}_{s,t}$. Therefore $(\Omega^\flat, \m{F}^\flat, P^\flat)$ is a lower bound of $\{(\Omega_I, \m{F}_I, \mu_I), T_{I,J}\}$, which means that the projective system $\{(\Omega_I, \m{F}_I, \mu_I), T_{I,J}\}$ over  $(\mathcal{K}_{s,t},\subseteq)$ is convergent.
 
 Conversely, suppose $\cali{S}$ admits a flow system $((\Omega^*, \m{F}^*, P^*), \{X_{s,t}^*\})_{\cali{S}}$. It follows that $(\Omega^*, \m{F}^*, P^*)$ is a lower bound for the projective system  $\{(\Omega_I, \m{F}_I, \mu_I), X_{I,J}\}$ over  $(\mathcal{K},\subseteq)$, with transformations $X_I^*: (\Omega^*, \m{F}^*, P^*)\to (\Omega_I, \m{F}_I, \mu_I)$ defined by $X_I^*(\omega)=( X_{\iota_0, \iota_1}^*(\omega), X_{\iota_1,\iota_2}^*(\omega),\cdots, X_{\iota_m,\iota_{m+1}}^*(\omega)),$ for all $\omega\in \Omega^*$ and $I=\{s=\iota_0<\iota_1<\,\dots<\iota_{m+1}=t\}\in\m{K}$. This assures the $K$-convergence of $\cali{S}$.
\end{proof}

\begin{observation}Convolution systems arising from random measures, such as those in Example \ref{exa3}, are always $K$-convergent because they admit flow systems. In contrast, Example \ref{exa1} illustrates both $K$-convergent and non-$K$-convergent systems.

For instance, consider $S=(\mathbb{Z}/m\mathbb{Z},+)$ equipped with the uniform distribution $\mu_m$, as in Observation \ref{obsa1}. The associated convolution system  $\mathcal{S}_{\mu_m}$ over $\mathbb{S}=(\mathbb{R},\leq)$ is $K$-convergent by Choksi's theorem \cite[Theorem 2.2]{Chol}, yet it does not admit an a.e.\ flow system on a standard Borel space; in fact, its measurable space $(\Omega^\flat, \mathcal{F}^\flat)$ is non-standard.

On the other hand, let \(
  S=\prod_{n=1}^{\infty}\mathbb{Z}/m\mathbb{Z},
\)
with coordinate-wise addition modulo $m$. Endow 
$S$ with the box topology, making it non-compact, and let 
$\m{G}$ be the associated Borel $\sigma$-field, which in this case is the entire power set of $S$. Define  $\mu=\bigotimes_{n=1}^\infty\mu_m$,
which is an idempotent probability measure on $(S, \m{G})$. The associated convolution system  $\cali{S}_\mu$ over $\mathbb{S}=(\mathbb{R},\leq)$
is easily seen to be convergent -- all finite-partition marginals glue to a genuine measure -- but it cannot admit a flow system (and hence fails to be 
$K$-convergent), because any putative flow would require uncountably many independent copies of 
$\mu$ (one per disjoint interval), which cannot coexist on a single probability space.

Further ``exotic'' non-$K$-convergent examples can be built adapting to our setting the techniques of \cite{TV,Tsi03,Tsi04}; these give rise to non-Fock Arveson product systems.  We will address these constructions elsewhere.
\end{observation}
Next, we focus on finding a necessary and sufficient condition for a convergent convolution system to be $K$-convergent. For this purpose, we start with a convergent convolution system $\cali{S}=\{\Omega_{s,t}, \m{F}_{s,t}, \mu_{s,t}, T_{r,s,t}\}_\mathbb{S}$, and let $\cali{S}^\flat= \{\Omega_{s,t}^\flat, \m{F}_{s,t}^\flat, \mu_{s,t}^\flat, T_{r,s,t}^\flat\}_\mathbb{S}$ be the projective CPPS associate with $\cali{S}$.  For any elements $u< s< t<v$ in $\mathbb{S}$, let $T^\flat_{(s,t),(u,v)}: \Omega^\flat_{u,v}\to \Omega^\flat_{s,t}$ be the map $$T^\flat_{(s,t),(u,v)}= \pi_{(s,t)} \comp (\operatorname{id}_{\Omega^\flat_{u,s}}\times T^\flat_{s,t,v})^{-1}\comp (T^{\flat}_{u,s,v})^{-1},$$ where $\pi_{(s,t)}: \Omega^\flat_{u,s}\times \Omega^\flat_{s,t} \times \Omega^\flat_{t,v}\to \Omega^\flat_{s,t}$ is the coordinate projection onto $\Omega^\flat_{s,t}$. If $u=s$, and/or $t=v$, then $T^\flat_{(s,t),(u,v)}$ is defined accordingly. It is clear that $T^\flat_{(s,t),(u,v)}: (\Omega^\flat_{u,v}, \m{F}_{u,v}^\flat, \mu_{u,v}^\flat)\to (\Omega^\flat_{s,t}, \m{F}_{s,t}^\flat, \mu_{s,t}^\flat )$ is a surjective transformation of probability spaces. Furthermore, we have: 
 
 \begin{lemma}\label{ll1}For any elements $u< s< t<v$ in $\mathbb{S}$ and partitions $I\in\m{K}_{s,t}$ and $J\in\m{K}_{u,v}$ we have $T^\flat_IT^\flat_{(s,t), (u,v)}=X_{I,J}T^\flat _J$.
 \end{lemma}
\begin{proof}By decomposing $J=I_L\cup \pe{I}\cup I_R$ as in (\ref{abba}) , we have
\begin{eqnarray*}
T^\flat_IT^\flat_{(s,t), (u,v)}&=&T^\flat_I\,\pi_{(s,t)} (\operatorname{id}_{\Omega^\flat_{u,s}}\times T^\flat_{s,t,v})^{-1}(T^{\flat}_{u,s,v})^{-1}\\&=&T_{I, \pe{I}}T^\flat_{\pe{I}}\pi_{(s,t)} (\operatorname{id}_{\Omega^\flat_{u,s}}\times T^\flat_{s,t,v})^{-1}(T^{\flat}_{u,s,v})^{-1}\\
&=&T_{I, \pe{I}}\pi_{\pe{I},J}(T^\flat_{I_R}\times T^\flat_{\pe{I}}\times T^\flat_{I_L})(\operatorname{id}_{\Omega^\flat_{u,s}}\times T^\flat_{s,t,v})^{-1}(T^{\flat}_{u,s,v})^{-1}\\
&=&X_{I, J}(T^\flat_{I_R}\times T^\flat_{\pe{I}}\times T^\flat_{I_L})(\operatorname{id}_{\Omega^\flat_{u,s}}\times T^\flat_{s,t,v})^{-1}(T^{\flat}_{u,s,v})^{-1}\\
&\stackrel{(\ref{sf})}{=}&X_{I,J}T^\flat _J,
\end{eqnarray*} as needed.
\end{proof} 
 
\begin{proposition}\label{projint} Let  $\cali{S}=\{\Omega_{s,t}, \m{F}_{s,t}, \mu_{s,t}, T_{r,s,t}\}_\mathbb{S}$ be a convergent convolution system. The system $(\{(\Omega_{s,t}^\flat, \m{F}_{s,t}^\flat, \mu_{s,t}^\flat)\}, T_{(s,t),(u,v)})$ is a projective system of probability spaces over the partially ordered set of ordered pairs $\m{C}=\{(s,t)\,|\,s,\,t\in \mathbb{S},\,s\leq t\}$ with partial order $(s,t)\subseteq (u,v)$ if $u\leq s\leq t\leq v$. 
\end{proposition}
\begin{proof} The compatibility relation $T^\flat_{(q,r),(s,t)}T^\flat_{(s,t), (u,v)} = T^\flat_{(q,r),(u,v)}$, for all $(q,r)\subseteq (s,t)\subseteq (u,v)$,
can be checked by using the laws of associativity. The process is, however, lengthy and convoluted. We will highlight the key intermediate steps and leave the detailed completion to the reader.
\begin{eqnarray*}
 T^\flat_{(q,r),(s,t)}T^\flat_{(s,t), (u,v)} 
          &=&  \pi_{(q,r)}\comp(T^\flat_{s,q,r}\times \operatorname{id}_ {\Omega^\flat_{r,t}})^{-1} \comp (T^{\flat}_{s,r,t})^{-1}\comp\\ && \pi_{(s,t)}\comp (\operatorname{id}_{\Omega^\flat_{u,s}}\times T^\flat_{s,t,v})^{-1}\comp (T^{\flat} _{u,s,v})^{-1}
          \\&=& \pi_{(q,r)}\comp( T^\flat_{s,q,r}\times \operatorname{id}_ {\Omega^\flat_{r,t}})^{-1} \comp (T^{\flat}_{s,r,t})^{-1}\comp\\
           && T^\flat_{s,r,t}\comp \pi_{(s,r)\times (r,t)}
           \comp (\operatorname{id}_{\Omega^\flat_{u,s}} \times \operatorname{id}_{\Omega^\flat_{s,r}}\times T_{r,t,v})^{-1}\comp\\
           && (\operatorname{id}_{\Omega^\flat_{u,s}}\times T^\flat_{s,r,v})^{-1}\comp (T^{\flat} _{u,s,v})^{-1}\\&=&
  \pi_{(q,r)} \comp (\operatorname{id}_{\Omega^\flat_{u,s}}\times  \operatorname{id}_ {\Omega^\flat_{s,q}}\times  \operatorname{id}_ {\Omega^\flat_{q,r}}\times T_{r,t,v}^\flat)^{-1}\comp\\&& (\operatorname{id}_{\Omega^\flat_{u,s}}\times T^\flat_{s,q,r}\times  \operatorname{id}_ {\Omega^\flat_{r,v}})^{-1}\comp (\operatorname{id}_{\Omega^\flat_{u,s}}\times T^\flat_{s,r,v})^{-1}\comp ( T^{\flat} _{u,s,v})^{-1}\\  
   &=&    \pi_{(q,r)} \comp (T^\flat_{u,s,q}\times  \operatorname{id}_ {\Omega^\flat_{q,r}}\times  \operatorname{id}_ {\Omega^\flat_{r,v}})\comp(\operatorname{id}_{\Omega^\flat_{u,s}}\times T^\flat_{s,q,r}\times  \operatorname{id}_ {\Omega^\flat_{r,v}})^{-1}\comp\\&& (\operatorname{id}_{\Omega^\flat_{u,s}}\times T^\flat_{s,r,v})^{-1}\comp (T^{\flat} _{u,s,v})^{-1}                        
          \end{eqnarray*}
          where $\pi_{(s,r)\times (r,t)}: \Omega^\flat_{u,s}\times \Omega^\flat_{s,r} \times \Omega^\flat_{r,t}\times \Omega^\flat_{t,v}\to \Omega^\flat_{s,r}\times \Omega^\flat_{r,t}$ is the coordinate projection. 
Noticing that $(\operatorname{id}_{\Omega^\flat_{u,q}}\times\circ T^{\flat}_{q,r,v})^{-1} \comp T^{-\flat}_{u,q,v}= (T^\flat_{u,s,q}\times  \operatorname{id}_ {\Omega^\flat_{q,r}}\times  \operatorname{id}_ {\Omega^\flat_{r,v}})\comp(\operatorname{id}_{\Omega^\flat_{u,s}}\times T^\flat_{s,q,r}\times  \operatorname{id}_ {\Omega^\flat_{r,v}})^{-1}\comp (\operatorname{id}_{\Omega^\flat_{u,s}}\times T^\flat_{s,r,v})^{-1}\comp T^{-\flat} _{u,s,v}$, it follows that $T^\flat_{(q,r),[s,t]} T^\flat_{(s,t),(u,v)} = T^\flat_{(q,r),(u,v)}$, as required.          
\end{proof}
We conclude with the following general description of the $K$-convergence of a convergent convolution system.
\begin{theorem}\label{Kimpa} Let $\cali{S}=\{\Omega_{s,t}, \m{F}_{s,t}, \mu_{s,t}, T_{r,s,t}\}_\mathbb{S}$ be a convergent convolution system such that the projective system of sets $\{\Omega_I, X_{I,J}\}$ over $\mathcal{K}$ is simply maximal. The following conditions are equivalent: \begin{enumerate}\item The convolution system $\cali{S}$ is $K$-convergent;
\item The projective system  $\{(\Omega_{s,t}^\flat, \m{F}_{s,t}^\flat, \mu_{s,t}^\flat), T_{(s,t),(u,v)}\}$ is convergent.
\end{enumerate}
\end{theorem}
\begin{proof}
(1) $\Rightarrow$ (2). Suppose that $\cali{S}$ is $K$-convergent. Because $T_{I,J}X_J=X_I$, for any two elements $s<t$ in $\mathbb{S}$ and any two partitions $I,\,J\in\m{K}_{s,t}$, $I\subseteq J$, it follows that there exists a surjective transformation of probability spaces $X_{s,t}^\flat:  (\Omega^\flat, \m{F}^\flat, P^\flat)\to (\Omega^\flat_{s,t}, \m{F}_{s,t}^\flat, \mu_{s,t}^\flat)$ such that \begin{eqnarray*}\label{ess}T^\flat_IX_{s,t}^\flat=X_I,\end{eqnarray*} for all $I\in\m{K}_{s,t}$ and all $s<t$ in $\mathbb{S}$. Using Lemma \ref{ll1}, one can immediately see that $T^\flat_{(s,t), (u,v)}X^\flat_{u,v}=X^\flat_{s,t}$, for all $u< s< t<v$ in $\mathbb{S}$. Indeed, $$T^\flat_IT^\flat_{(s,t), (u,v)}X^\flat_{u,v}=X_{I,J}T^\flat _JX^\flat_{u,v}=X_{I,J}X_J=X_I,$$
for all $I\in\m{K}_{s,t}$, $J\in \m{K}_{u,v}$, $I\subseteq J$, and the conclusion follows.
 Consequently $(\Omega^\flat, \m{F}^\flat, P^\flat)$ is an upper limit for the projective system $\{(\Omega_{s,t}^\flat, \m{F}_{s,t}^\flat, \mu_{s,t}^\flat), T_{(s,t),(u,v)}\}$, so this must be convergent. 
 
(2) $\Rightarrow$ (1). Suppose that the projective system $\{(\Omega_{s,t}^\flat, \m{F}_{s,t}^\flat, \mu_{s,t}^\flat), T_{(s,t),(u,v)}\}$ is convergent, and let \begin{eqnarray}\label{hiis}
({}^\flat\Omega, {}^\flat\m{F}, {}^\flat P)=\limproj_{\m{C}}\{(\Omega_{s,t}^\flat, \m{F}_{s,t}^\flat, \mu_{s,t}^\flat), T_{(s,t),(u,v)}\}\end{eqnarray} be its projective limit with coordinate projections $T'_{s,t}:({}^\flat\Omega, {}^\flat\m{F}, {}^\flat P)\to (\Omega_{s,t}^\flat, \m{F}_{s,t}^\flat, \mu_{s,t}^\flat)$, for all $s<t$ in $\mathbb{S}$. For any $I\in\m{K}$, say $I\in\m{K}_{s,t}$, for some $s<t$ in $\mathbb{S}$, consider the composite surjective transformation
$T_I^\flat T'_{s,t}: ({}^\flat\Omega, {}^\flat\m{F}, {}^\flat P)\to (\Omega_I, \m{F}_I, \mu_I)$. We deduce from Lemma \ref{ll1} that $T_I^\flat T'_{s,t}=X_{I, J}T^\flat _JT'_{u,v}$
for all $I\subseteq J$, $I\in \mathcal{K}_{s,t}$, $J\in\mathcal{K}_{u,v}$, $u\leq s< t\leq v$. Thus $({}^\flat\Omega, {}^\flat\m{F}, {}^\flat P)$ is a lower bound for $\{(\Omega_I, \m{F}_I, \mu_I), X_{I,J}\}$ over  $(\mathcal{K},\subseteq)$, which leads to the convergence of the projective system  $\{(\Omega_I, \m{F}_I, \mu_I), X_{I,J}\}$.
\end{proof}
%\begin{remark}
It is worth mentioning that if $\cali{S}=\{\Omega_{s,t}, \m{F}_{s,t}, \mu_{s,t}, T_{r,s,t}\}_\mathbb{S}$ is a $K$-convergent convolution system, then the probability space $ ({}^\flat\Omega, {}^\flat\m{F}, {}^\flat P)$ defined in (\ref{hiis}) is the greatest lower bound for the projective system $\{(\Omega_I, \m{F}_I, \mu_I), X_{I,J}\}$ over  $(\mathcal{K},\subseteq)$. We postpone the proof of this fact, as well as the discussion of some relevant examples, for another occasion.

%\subsection*{Acknowledgements:} The authors would like to thank the anonymous reviewers for their comments and suggestions, which helped improve the clarity of the paper.

\end{document}